\documentclass[12pt]{article}
\parskip=\medskipamount
\usepackage{amsmath,amsthm,amssymb}
\usepackage{graphicx,epsfig}

\numberwithin{equation}{section}

\newtheorem{theorem}{Theorem}%[section]
\newtheorem{lemma}{Lemma}
\newtheorem{proposition}{Proposition}

\def\ve{\varepsilon}
\def\vp{\varphi}
\def\arrowk{^\to{\kern -6pt\topsmash k}}
\def\arrowK{^{^\to}{\kern -9pt\topsmash K}}
\def\arrowt{^\to{\kern -6pt\topsmash t}}
\def\arrowr{^\to{\kern-6pt\topsmash r}}
\def\arrowvp{^\to{\kern -8pt\topsmash\vp}}
\def\tk{\tilde{\kern 1 pt\topsmash k}}
\def\barm{\bar{\kern-.2pt\bar m}}
\def\barN{\bar{\kern-1pt\bar N}}
\def\barA{\, \bar{\kern-3pt \bar A}}

\def\be{\begin{equation}}
\def\ee{\end{equation}}
\usepackage{amsmath,amssymb,amsbsy,amsfonts,amsthm,latexsym,
                amsopn,amstext,amsxtra,euscript,amscd}

\def\ve{\varepsilon}

\def\iint{\not\kern -4pt\int}
\def\iiint{{\small{_\not}}\kern-3.5pt\int}
\def\mod{\text{\rm mod }}
\def\be{\begin{equation}}
\def\ee{\end{equation}}
\numberwithin{equation}{section}

\begin{document}
\date{}
\title
{Moebius Schr\"odinger}
\author
{{J.~Bourgain}\\
{Institute for Advanced Study}\\
{1 Einstein Drive}\\
{Princeton, NJ 08540}}
\maketitle

\noindent
{\bf (1).} Let $\mu (n)$ be the Moebius function and consider the Schr\"odinger operator on $\mathbb Z_+$
$$
H=\Delta+\lambda\mu \qquad (\lambda\not=0 \text { arbitrary}).\eqno{(1.0)}
$$
We prove the following
\medskip

\noindent
\begin{theorem}\label{Theorem1} {\it For $E\in\mathbb R$ outside a set of 0-measure, any solution $\psi=(\psi_n)_{n\geq 0}$, $\psi_0=0, \psi\not= 0$  of
$$
H\psi = E\psi
$$
satisfies
$$
\overline{\lim} \ \frac{\log^+|\psi_n|}n >0.\eqno{(1.1)}
$$}
\end{theorem}

Recalling the spectral theory of 1D Schr\"odinger operators with a random potential, Theorem \ref{Theorem1} fits the general heuristic, known as the
`Moebius randomness law' (cf. \cite{Sa}).
The question whether (1.0) satisfies Anderson localization remains open and is probably difficult.

The fact that $H$ has no ac-spectrum is actually immediate from the following result of Remling.

\begin{proposition}\label{Proposition1} \text {\rm ([R], Theorem 1.1)}:
Suppose that the (half line) potential $V(n)$ takes only finitely many values and $\sigma_{ac} \not= \phi$.
Then $V$ is eventually periodic.
\end{proposition}

We will use again Proposition 1 later on, in the proof of the Theorem.

\bigskip

\noindent
{\bf (2).}  Let $X\subset \{0, 1, -1\}^{\mathbb Z}$ be the point-wise closure of the set $\{T^j\overline\omega; j\in\mathbb Z\}$, where $T$ is the left shift and
$\overline\omega$ defined by
$$
\overline\omega_n= \begin{cases} \mu(n) \text { for } n\in\mathbb Z_+\\ 0\text { for } n\in\mathbb Z_-.\end{cases}\eqno{(2.1)}
$$
Let
$$
\nu_N =\frac 1N\sum_{j=0}^{N-1} \delta_{T^j \overline{\omega}} \qquad (\delta_x =\text { Dirac measure at $x$)}
$$
and $\nu \in\mathcal P(X)$ a weak$^*$-limit point of $\{\nu_N\}$.

Then $\nu$ is a $T$-invariant probability measure on $X$.

The only property of the Moebius function exploited in the proof of Theorem \ref{Theorem1} is the following fact.

\begin{lemma} \label{Lemma2}
For no element $\omega\in X, (\omega_n)_{n\geq 0}$ is eventually periodic, unless $\omega_n=0$ for $n$ large enough. Similarly for $(\omega_n)_{n\leq 0}$.
\end{lemma}

\begin{proof}
Suppose $\omega$ eventually periodic. Hence there is $n_0\in\mathbb Z_+$ and $d\in\mathbb Z_+$ such that
$$
\omega(n+d) =\omega(n) \text { for } n\geq n_0.\eqno{(2.2)}
$$
Take $N=10^3 (n_1+d^3)$ and choose $n_1 \geq n_0$ and $k\in\mathbb Z_+$ such that
$$
\omega(n) =\mu(k+n) \text { for } n\in [n_1, n_1+N].\eqno{(2.3)}
$$
Let $d<p<10d$ be a prime.
Taking $n\in [n_1, n_1+ \frac N2]$, there is $0\leq j<p^2$ such that $k+n+jd\equiv 0 (\mod p^2)$ and thus $\mu(k+n+jd)=0$.
Since $n+jd\in [n_1, n_1+N]$, (2.3), (2.2) imply that $\mu(k+n+jd)=\omega(n+jd)=\omega(n)$ and therefore $\omega =0$ on $[n_1, n_1+\frac N2]$, hence on
$[n_1, \infty [ $.\end{proof}

Denote for $\omega\in X$
$$
H_\omega =\Delta+\lambda\omega.\eqno{(2.4)}
$$
Combined with Proposition \ref{Proposition1}, Lemma \ref{Lemma2} implies

\begin{lemma}\label{Lemma3}
$$
\sigma_{ac} (H_\omega) =\phi\qquad (\nu  - \text {a.e.})
$$
\end{lemma}

\begin{proof}
Denoting $H^\pm_{\omega}$ the corresponding halfline $SO$'s, we have
$$
\sigma_{ac} (H_\omega)=\sigma_{ac} (H^+_\omega)\cup\sigma _{ac} (H^-_\omega)
$$
and these sets are empty, unless
$$
\omega\in \bigcup^\infty_{k=1} \{\omega\in X; \omega_n=0 \text { for all } n\geq k \text{ or all } n\leq - k\}.\eqno{(2.5)}
$$
Clearly $\nu \, (2.5)=0$.
\end{proof}

The measure $\nu$ need not be $T$-ergodic, so we consider its ergodic decomposition
$$
\nu =\int\nu_\alpha d\alpha.\eqno{(2.6)}
$$
For each $\alpha$, let $\gamma_\alpha(E)$ be the Lyapounov exponent of $H_\omega$, i.e.
$$
\gamma_\alpha(E) =\lim_{N\to\infty} \frac 1N\log \Big\Vert \prod^0_N\begin{pmatrix} E-\lambda\omega_n&-1\\ 1& 0\end{pmatrix}\Big\Vert
\quad (\nu_\alpha \, a.e).\eqno{(2.7)}
$$

Next, we apply Kotani's theorem (for stochastic Jacobi matrices, as
proven in \cite {Si}, Theorem 2).

\begin{proposition}\label{Proposition4}
(assuming $(\Omega,\mu, T)$ ergodic).

If $\gamma(E)=0$ on a subset $A$ of $\mathbb R$ with positive Lebesque measure, then $E_\omega^{ac} (A) \not= 0$ for a.e. $\omega$.
\end{proposition}
\medskip

($E^{ac}$ denote the projection on the $ac$-spectrum).
\medskip

Apply Proposition \ref{Proposition4}  to $H_\omega$ on $(X, \nu_\alpha)$.
By Lemma \ref{Lemma3}, $E_\omega^{ac}=0$, $\nu_\alpha$ a.e., hence $\{E\in\mathbb R; \gamma_\alpha(E)=0\}$ is a set of zero Lebesgue measure.
For $E$ outside a subset $\mathcal E_*\subset\mathbb R$ of zero Lebesque measure, we have that $\gamma_\alpha(E)>0$ for almost all $\alpha$ in (2.6), therefore
$$\begin{aligned}
&\liminf_{N \rightarrow \infty} \int \frac 1N \log\Big\Vert \prod^0_N\begin{pmatrix} E-\lambda\omega_n &-1\\ 1& 0\end{pmatrix} \Big\Vert\nu (d\omega)\geq\\
&\int\Big\{\liminf_{N \rightarrow \infty} \int \Big[\frac 1N
\log\Big\Vert\prod^0_N\begin{pmatrix} E-\lambda\omega_n & -1\\
1&0\end{pmatrix} \Big\Vert\Big] \nu_\alpha (d\omega) \Big\}
d\alpha\geq
\end{aligned}
$$
$$
\qquad\qquad \int\gamma_\alpha (E) d\alpha>0.
\eqno{(2.8)}
$$

Denoting $R_N$ the restriction operator to $[1, N]$, let
$$
H_\omega^{(N)} = R_NH_\omega R_N
$$
$$
G_\omega^{(N)} (E)= (H^{(N)}_\omega -E+i0)^{-1} \qquad \text {(= restricted Green's function)}.
$$

Recall that by Cramer's rule, for $1\leq k_1\leq k_2\leq N$
$$
|G_\omega^{(N)} (E) (k_1, k_2)| = \frac {\det [H^{(k_1-1)}_\omega -E] .|\det [H_{T^{k_2}\omega}^{(N-k_2)} -E]|}
{|\det [H_\omega^{(N)} -E]|}\eqno{(2.9)}
$$
and also the formula
$$
\begin{aligned}
M_N(E, \omega)&=\prod^1_N \begin{pmatrix} E-\lambda\omega_n&-1\\ 1&0\end{pmatrix}\\[7pt]
&= \left
[\begin{matrix} \det [E-H_\omega^{(N)}] & -    \det [E-H_{T_\omega}^{(N-1)}\\[5pt]
                \det [E-H_{\omega}^{(N-1)}] & -\det [E-H_{T_\omega}^{(N-2)}]\end{matrix}\right].
\end{aligned}
\eqno{(2.10)}
$$
Using the above formalism, it is well-known how to derive from positivity of the Lyapounov exponent, bounds and decay estimates on the restricted Green's
functions.
Since ergodicity of the measure is used, application to the preceding requires to start from the $\nu_\alpha$.

For $E\in\mathbb R, \delta, c>0, M\in \mathbb Z_+$, define
$$
\begin{aligned}
\Omega_{E, \delta, c, M}&=\{\omega\in X; \Vert G_\omega^{(M)} (E)\Vert< e^{\delta M}
 \text { and } |G_\omega^{(M)} (E) (k, k')|< e^{-c|k-k'|}\\[10pt]
& \text { if } 1\leq k, k'\leq M \text { and } |k-k'|>\delta M\}.
\end{aligned}
\eqno{(2.11)}
$$

Fix $\alpha$ and $\delta>0$.
Then $E$ a.e
$$
\lim_{M\to\infty} \nu_\alpha(\Omega_{E, \delta, \frac 12\gamma_\alpha (E), M})=1.\eqno{(2.12)}
$$
Using Fubini arguments and (2.6), we derive the following

\begin{lemma}\label {Lemma5}
Given $\ve>0$, there is $b>0$, such that for all $\delta>0$, there is a subset $\mathcal E_\ve\subset\mathbb R$, mes\,$\mathcal E_\ve<\ve$ and some scale $M$
satisfying
$$
\nu(\Omega_{E, \delta, b, N})> 1-\ve \text { for } E\not\in \mathcal E_\ve \text{ and $N>M$}.\eqno{(2.13)}
$$
\end{lemma}

\noindent
{\bf(3).}
Using the definition of $\nu$, we re-express (2.13) in terms of the Moebius function.

Let $H$ be as in (1.0).
For $I\subset \mathbb Z_+$ an interval, denote
$$
H_I=R_IHR_I\eqno{(3.1)}
$$
and
$$
G_I(E)=(H_I-E+io)^{-1}.\eqno {(3.2)}
$$
Let $S=S_{E, \delta, N}$ be defined by
$$
\begin{aligned}
S&=\{k\in\mathbb Z; \Vert G_{[k, k+N[} (E)\Vert < e^{\delta N} \text {and } \\[10pt]
&|G_{[k, k+N[}(E) (k', k'')|
< \overline e^{b|k'-k''|} \text { if } k\leq k', k''\leq k +N, |k'-k''|> \delta
N.
\end{aligned}
\eqno{(3.3)}
$$

Property (2.13) then translates as follows
$$
\lim_{\substack{\ell\to \infty\\ \ell\gg N}} \  \frac 1\ell \ |S\cap [1, \ell]|>\frac 12\eqno{(3.4)}
$$
for $E\not\in\mathcal E_\ve$ and $N>M$.
Here `lim' refers to the Banach limit in the definition of $\nu$.

Fix $\ve>0$ a small number, take $0< b<\frac 1{10}$ as in Lemma \ref{Lemma5} and let $\delta =b^{10}$.
Let $\mathcal E_\ve\subset\mathbb R, M>\delta^{-2} +\frac 1\ve$, satisfy the lemma. Hence, from (3.4)
$$
\lim_{\substack{\ell\to\infty\\ \ell\gg M}} \ \frac 1\ell \ |S_{E, \delta, M} \cap [1, \ell]|>\frac 12 \text { for } E\not\in \mathcal E_\ve.\eqno{(3.5)}
$$
Choose $\ell\gg M$ such that
$$
\frac 1\ell |S_{E, \delta, M}\cap [1, \ell]|>\frac 12\text { for } E\not\in \mathcal E_\ve'
$$
where $\mathcal E_\ve \subset\mathcal E_\ve'\subset\mathbb R$ satisfies
$$
\text{mes\,} \mathcal E_\ve'< 2\ve.\eqno{(3.7)}
$$
Next we rely on a construction from \cite {B}, Lemma 6.1 and Corollary 6.54.  We recall the statement

\begin{lemma}\label{Lemma6}
Let $\underline{0< c_0< 1}$, $\underline{0 < c_1< \frac 1{10}}$ be constants, $\underline{0<\delta<c_1^{10}}$ and $\ell\gg M>\delta^{-2}$.

Let
$$
A=v_n\delta_{nn'} +\Delta \quad (1\leq n, n'\leq \ell)\eqno{(3.8)}
$$
(hence $A$ is an $\ell\times\ell$ matrix) with diagonal $v_n$ arbitrary, bounded, $|v_n|=0(1)$.

Let $\mathcal U\subset\mathbb R$ be a set of energies $E$ such that for each $E\in \mathcal U$, the following holds:

There is a collection $\{I_\alpha\}$ of disjoint intervals in $[1, \ell], |I_\alpha|=M$ such that for each $\alpha$
$$
\Vert(R_{I_\alpha}(A-E) R_{I_\alpha})^{-1}\Vert< e^{\delta M}\eqno{(3.9)}
$$
and
$$
|(R_{I_\alpha} (A-E) R_{I_\alpha})^{-1} (k, k')|< e^{-c_1|k-k'|} \text { for } k, k' \in I_\alpha, |k-k'|> \delta M\eqno{(3.10)}
$$
holds, and
$$
\sum_\alpha |I_\alpha|> c_0 \ell .\eqno{(3.11)}
$$
Then there is a set $\mathcal E''\subset\mathbb R$ so that
$$
\text{mes\,} (\mathcal E'')<\frac 1M\eqno{(3.12)}
$$
and for $E\in \mathcal U\backslash \mathcal E''$,
$$
\max_{\substack{1\leq x\leq \frac{c_0}{10} \ell\\ \ell\geq y\geq \ell -\frac {c_0}{10}\ell}}
|(A-E)^{-1} (x, y)|< e^{-\frac 18 c_0c_1\ell}.\eqno{(3.13)}
$$
\end{lemma}

The proof of Lemma \ref{Lemma6} is a bit technical, but uses nothing
more than the resolvent identity and energy perturbation.

Let $v_n=\lambda\mu (n)$.

Take $c_0=\frac 12, c_1=b, \mathcal U=\mathbb R\backslash \mathcal E_\ve'$ with $\mathcal E_\ve'$ as above:

Let $\ell_0 \gg M$ satisfy (3.6).
From the definition (3.3) of $S_{E, \delta, M}$ and (3.6), we clearly obtain a collection $\{I_\alpha\}$ of $M$-intervals in $[1, \ell]$ such that
(3.9)-(3.11) hold.

It follows that for $E$ outside of the set $\mathcal E_\ve''=\mathcal E_\ve'\cup\mathcal E''$ of measure at most $2\ve+\frac 1M<3\ve$, one has for $b'\sim b$ that
$$
\max_{\substack {1\leq x\leq \frac{c_0}{10} \ell\\ \ell \geq y\geq \ell -\frac {c_0}{10}\ell}}|G_{[1, \ell]} (E) (x, y)|<e^{-b'\ell}.\eqno{(3.14)}
$$
Note that $b'>0$ depends on $\ve$ and $\nu$ and $\mathcal E_\ve''$ depends on $\ell$, which can be taken arbitrarily large in the subsequence of $\mathbb Z_+$
used to define $\nu$.  Since this subsequence is arbitrary, it follows that there is some $b'=b_\ve$ and $\ell_\ve\in\mathbb Z_+$
such that for $\ell> \ell_\ve$
$$
\text{mes\,} [E\in\mathbb R; \max_{\substack { 1\leq x\leq \frac{c_0}{10}\ell\\ \ell\geq y\geq \ell -\frac{c_0}{10}\ell}} |G_{[1, \ell]} (E) (x, y)| >
e^{-b'\ell}]=\text { mes\,} \tilde{\mathcal E}_\ell < \ve.
\eqno{(3.15)}
$$
Assume $\psi= (\psi_n)_{n\geq 0}, \psi_0=0$ a solution of
$$
H\psi = E\psi.
$$
Taking $\ell$ large, one has by projection
$$
H_{[1, \ell]} \psi^{(\ell)}+ \psi_{\ell+1} e_\ell = E\psi^{(\ell)}\eqno{(3.16)}
$$
where $\psi^{(\ell)} =\sum_{1\leq x\leq \ell}\psi_x e_x, \{e_x\}$ the unit vector basis.

Hence
$$
\psi^{(\ell)} =- \psi_{\ell+1} G_{[1, \ell]} (E) e_\ell
$$
and fixing some coordinate $x\geq 1$, for $\ell$ large enough
$$
|\psi_x|\leq |\psi_{\ell+1}| \ |G_{[1, \ell]} (E) (x, \ell)|.\eqno{(3.17)}
$$
Take $x$ with $\psi_x\not= 0$.
Assuming
$$\overline{\lim_n} \ \frac{\log^+|\psi_n|}n=0
$$
it follows from (3.17) that
$$
\overline{\lim_\ell} \ \frac 1\ell \  \log^+ |G_{[1, \ell]} (E) (x, \ell)|^{-1} =0.\eqno{(3.18)}
$$
From the definition of $\tilde {\mathcal E}_\ell$ in (3.15), this means that
$$
E\in\bigcup_{\ell_0} \bigcap_{\ell\geq \ell_0} \tilde{\mathcal E_\ell}\eqno{(3.19)}
$$
which is a set of measure $\leq \ve$.

Letting $\ve\to 0$,  Theorem \ref{Theorem1} follows.

\noindent
{\bf (4).} Taking into account the comment made prior to Lemma \ref{Lemma2}, our argument gives the following more general result,
that can be viewed as a refinement of \cite{R}.

\begin{theorem} \label{Theorem2}
Suppose that the (half line) potential $(V_n)_{n\geq 0}$
takes only finitely many values and satisfies the following property
$$
\operatornamewithlimits{\overline\lim}\limits_{r \to\infty} \ \operatornamewithlimits{\overline\lim}\limits_{N\to \infty}
\frac 1N  | \{1\leq k\leq N; V_k=\omega_0, V_{k+1} =\omega_1, \ldots, V_{k+r}= \omega_r\}|=0
\eqno{(4.1)}
$$
whenever  $\overline\omega=(\omega_r)_{r\geq 0}$ is a periodic sequence in the pointwise closure of the sequences $(V_{n+j})_{n\in\mathbb Z_+}
\quad (j\in\mathbb Z_+)$.

Then the Schr\"odinger operator $H=\Delta+V$ satisfies the conclusion of Theorem \ref{Theorem1}.
\end{theorem}

\noindent {\bf Acknowledgement.} The author is grateful to P.~Sarnak
for bringing the problem to his attention and several discussions.


\begin{thebibliography}{AAA}

\bibitem [B]{B} J.~Bourgain, Positive Lyapounov exponents for most energies. {\emph LMN}, 1745, 37--66 (2000).

\bibitem [R]{R} C.~Remling, The absolutely continuous spectrum of Jacobi matrices.  Preprint.

\bibitem [Sa]{Sa} P.~Sarnak, Moebius randomness law. Notes.

\bibitem [Si]{Si} B.~Simon, Kotani theory for one-dimensional stochastic Jacobi matrices. {\emph CMP} 89, 227--234 (1983)
\end{thebibliography}
\end{document}